\documentclass{amsart}
\usepackage{amssymb, verbatim, epsfig}
\usepackage{graphicx}

\newcommand{\RR}{{\mathbb{R}}}

\newcommand{\NN}{{\mathbb{N}}}

\DeclareMathOperator{\sgn}{{sgn}}
\DeclareMathOperator{\supp}{{supp}}

\newcommand\Op{\operatorname{Op}}

\theoremstyle{plain}
\newtheorem{theorem}{Theorem}
\newtheorem{proposition}[theorem]{Proposition}
\newtheorem{lemma}[theorem]{Lemma}

\theoremstyle{definition}

\theoremstyle{remark}
\newtheorem{remark}[theorem]{Remark}

\begin{document}
\title[non-QUE ergodic billiards]{Ergodic billiards that are not quantum unique ergodic}

\author[Andrew Hassell]{Andrew Hassell \\ \\ with an appendix by Andrew Hassell and Luc Hillairet}
\address{Department of Mathematics, Australian National University, Canberra 0200 ACT, \phantom{ar} AUSTRALIA}
\email{hassell@maths.anu.edu.au}

\address{UMR CNRS 6629-Universit\'{e} de Nantes, 2 rue de la Houssini\`{e}re, \\
BP 92 208, F-44 322 Nantes Cedex 3, France}
\email{Luc.Hillairet@math.univ-nantes.fr}

\keywords{Eigenfunctions,  partially rectangular domains, Bunimovich stadium, bouncing ball modes, quantum unique ergodicity}
\subjclass[2000]{35P20, 58J50}
\thanks{This research was partially supported by Discovery Grant DP0771826
from the Australian Research Council}

\maketitle
\begin{abstract} Partially rectangular domains are compact two-dimensional Riemannian manifolds $X$, either closed or with boundary, that contain a flat rectangle or cylinder. In this paper we are interested in partially rectangular domains with ergodic billiard flow; examples are the 
Bunimovich stadium, the Sinai billiard or Donnelly surfaces. 

We consider a one-parameter family $X_t$ of such domains parametrized by the aspect ratio $t$ of their rectangular part. There is convincing
theoretical and numerical evidence that  the Laplacian on $X_t$ with Dirichlet or Neumann boundary conditions is not quantum unique ergodic (QUE). We prove that this is true for all $t \in [1,2]$ excluding, possibly, a set of Lebesgue measure zero. This yields the first examples of ergodic billiard systems proven to be non-QUE. 

\end{abstract}

\section{Introduction} 
A partially rectangular domain $X$ is a compact Riemannian 2-manifold, either closed or with boundary, that contains a flat rectangle or cylinder, in the sense that $X$ can be decomposed $X = R \cup W$, where $R$ is a  rectangle\footnote{For brevity we use `rectangle' to mean `rectangle or cylinder'}, $R = [-\alpha, \alpha]_x \times [-\beta, \beta]_y$ (with $y = \pm \beta$ identified in the case of a cylinder) carrying the flat metric $dx^2 + dy^2$, and such that $R \cap W  = R \cap \{ x = \pm\alpha \}$. 

The main result of this paper is that partially rectangular domains $X$ are usually not QUE (see Theorem~\ref{ae}). This is primarily of interest in the case that $X$ has ergodic billiard flow; examples include the Bunimovich stadium, the Sinai billiard, and Donnelly's surfaces \cite{Bun}, \cite{Sinai}, \cite{D} --- see Figure~\ref{f1}. Ergodicity implies that these domains are  \emph{quantum ergodic} by a theorem of G\'erard-Leichtnam \cite{GL} and Zelditch-Zworski \cite{ZZ}, generalizing work of Schnirelman \cite{S}, Zelditch \cite{Z} and Colin de Verdi\`ere \cite{CdV} in the boundaryless case. Quantum ergodicity is a statement about the eigenfunctions $u_j$ of the positive Laplacian $\Delta$ associated to the metric on $X$, where we assume that the Dirichlet or Neumann boundary condition is specified if $X$ has boundary. 
The operator $\Delta$ has a realization as a self-adjoint operator on $L^2(X)$ and has discrete spectrum $0 < E_1 < E_2 \leq E_3 \dots \to \infty$ and corresponding orthonormal real eigenfunctions $u_j$, unique up to orthogonal transformations in each eigenspace.

The statement that $\Delta$ is quantum ergodic is the statement that 
 there exists a density one set $J$ of natural numbers such that the subsequence $(u_j)_{j \in J}$ of eigenfunctions 
has the following equidistribution property: for each semiclassical pseudodifferential operator $A_h$, properly supported in the interior of $X$, we have
\begin{equation}
\lim_{j \in J \to \infty} \langle A_{h_j} u_j, u_j \rangle_{L^2(X)} = \frac1{|S^* X|} \int_{S^* X} \sigma(A).
\label{qe}\end{equation}
Here $h_j = E_j^{-1/2}$ is the length scale corresponding to $u_j$, $S^*X$ is the cosphere bundle of $X$ (the bundle of unit cotangent vectors), and $|S^*X|$ denotes the measure of $S^*X$ with respect to the natural measure induced by Liouville measure on $T^*X$. In particular, this holds when $A_h$ is multiplication by a smooth function $\eta$ supported in the interior $X^\circ$ of $X$. In that case, \eqref{qe} reads 
$$
\lim_{j \in J \to \infty} \int_X \eta \, u_j^2  \, dg  = \frac1{|X|} \int_{X} \eta \, dg
$$
which implies that the probability measures $u_j^2$ tend weakly to uniform measure on $X$ (for $j \in J$); the condition \eqref{qe} is a finer version of this statement that can be interpreted as equidistribution of the $u_j$, $j \in J$ in phase space. 
\emph{Quantum unique ergodicity} (QUE) is the stronger property that \eqref{qe} holds for the full sequence of eigenfunctions, i.e. that $J$ can be taken to be $\mathbb{N}$. 

These properties can also be expressed in terms of quantum limits, or semiclassical measures. These are measures  on $T^*X$ obtained as weak limits of subsequences of the measures $\mu_j$ which act on compactly supported functions on $T^*X^\circ$ according to 
$$
C_c^\infty (T^*X^\circ) \ni a(x, \xi) \mapsto \langle \Op_{h_j}(a) u_j, u_j \rangle.
$$
Here $\Op_h$ is a semiclassical quantization of the symbol $a$. Quantum ergodicity then is the statement that for a density one sequence $J$ of integers, the sequence $\mu_j$ converges weakly to uniform measure on $S^*X$, and QUE the same property with $J = \NN$. 

There are rather few results, either positive or negative, on quantum unique ergodicity. Rudnick-Sarnak \cite{RS} conjectured that 
closed hyperbolic manifolds are always QUE. This has been verified by Lindenstrauss, Silbermann-Venkatesh and Holowinsky-Soundararajan in some arithmetic cases  \cite{L} \cite{SV} \cite{SV2} \cite{HS}, provided one restricts to Hecke eigenfunctions which removes any eigenvalue degeneracy which might be present in the spectrum. In the  negative direction,  Faure-Nonnenmacher and 
De Bi\`evre-Faure-Nonnenmacher  \cite{FN}, \cite{FNDB} showed that certain quantized cat maps on the torus are non-QUE. In related work, Anantharaman \cite{A} has shown that quantum limits  on a closed, negatively curved manifold have positive entropy, which rules out quantum limits supported on a finite number of periodic geodesics.  
Up till now there have been no billiard systems rigorously proved to be either QUE or non-QUE. 

The results of the present paper are based crucially on the fact that partially rectangular domains $X$ may be considered part of a one-parameter family $X_t$ where we fix the height $\beta$ of the rectangle and vary the length $\alpha$. Here we arbitrarily set $\beta = \pi/2$ and let $\alpha = t \beta$ with $t \in [1,2]$. Our main result is 

\begin{theorem}\label{ae} For almost every value of $t \in [1,2]$, 
the Laplacian on  $X_t$ is non-QUE. 
\end{theorem}

This is proved first in the simple setting of the stadium billiard with the Dirichlet boundary condition. In the appendix, which is joint work with Luc Hillairet, we show how to obtain the result for any partially rectangular domain.

\begin{figure}\label{f1}
\includegraphics[scale=.2]{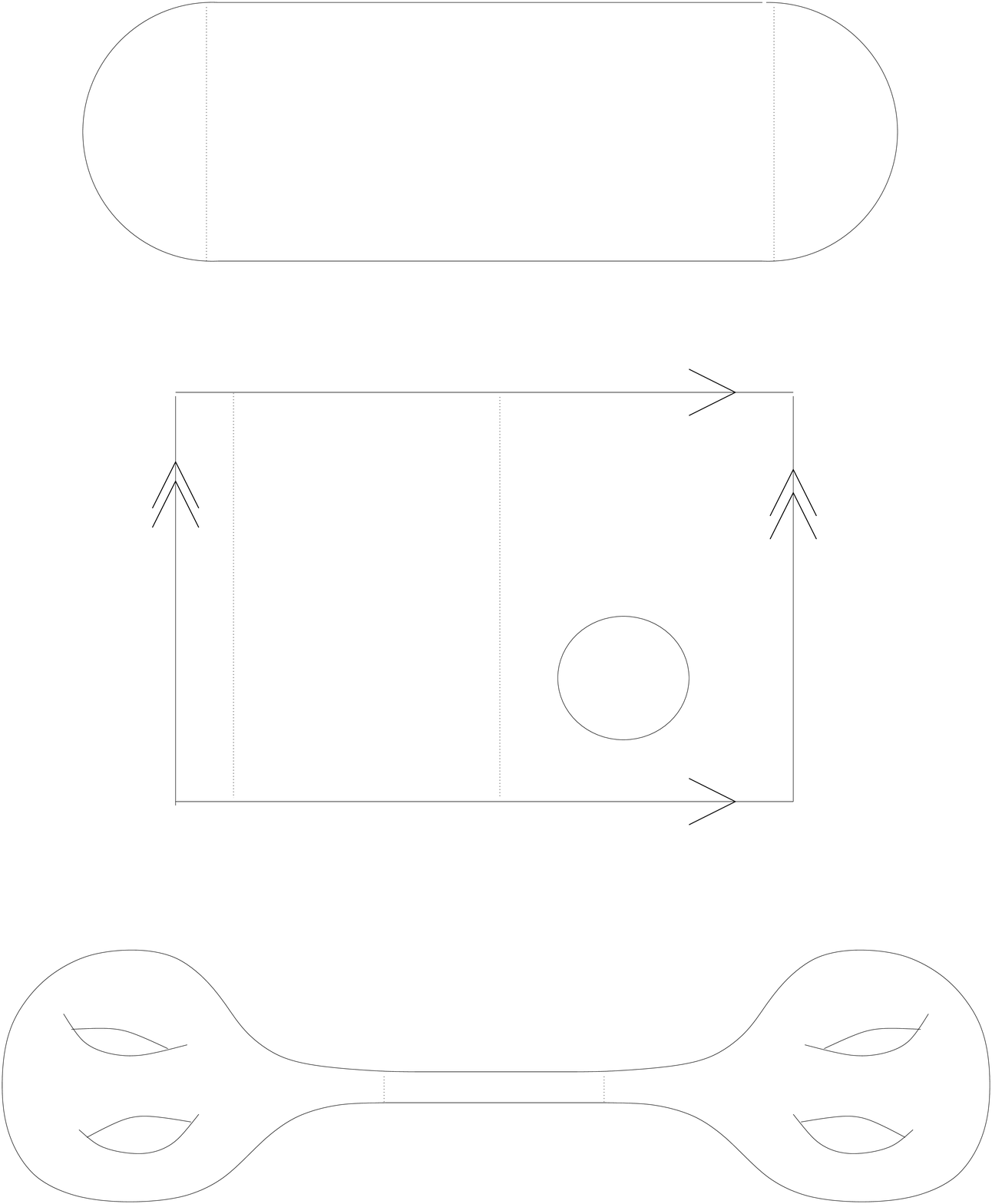}
\caption{Examples of ergodic, partially rectangular domains: the Bunimovich stadium (top), Sinai billiard (middle), Donnelly surface (bottom). }
\end{figure}

The proof is based on the original argument of Heller and O'Connor \cite{HOC} as refined by Zelditch \cite{Zque}, using `bouncing ball' quasimodes. Their argument shows that QUE fails provided that one can find a subsequence of intervals of the form $[n^2 - a, n^2 + a]$, for arbitrary fixed $ a >0$, such that the number of eigenvalues in this interval is bounded uniformly as $n \to \infty$ along this subsequence. Note that in two dimensions, the expected number of eigenvalues in the interval $[E-a, E+a]$ is independent of $E$, so this is a very plausible condition.

Let us recall this argument in more detail. For simplicity, suppose the Dirichlet boundary condition is imposed on the horizontal sides of the rectangle $R$. 
Consider the function $v_n \in \operatorname{dom}(\Delta)$  given by  $\chi(x) \sin ny$ for even $n$ and $\chi(x) \cos ny$ for odd $n$, where $\chi(x)$ is supported in $x \in [-\pi/4, \pi/4]$. (For other boundary conditions, we replace $\sin ny$ and $\cos ny$ by the corresponding one-dimensional eigenfunctions; in the cylindrical case, we use $e^{\pm i n y}$ and take $n$ even.) 
For convenience, we choose $\chi$ so that $\| v_n \|_{L^2} = 1$ for all $n$.  The $v_n$ are so-called `bouncing ball'  quasimodes; they concentrate semiclassically as $n \to \infty$ onto a subset of the bouncing ball trajectories, which are the periodic trajectories  that bounce vertically (i.e. with $x$ fixed) between the horizontal sides of the rectangle $R$.  They satisfy $\| (\Delta -  n^2) v_n \|_{L^2} \leq K$, uniformly in $n$. It follows from basic spectral theory that 
\begin{equation}
\| P_{[n^2 - 2K, n^2 + 2K]} v_n  \|^2 \geq \frac{3}{4}
\label{proj}\end{equation}
where $P_I$ is the spectral projection of the operator $\Delta$ corresponding to the set $I \subset \RR$. Suppose there exists a subsequence $n_j$ of integers with the property that there exists $M$, independent of $j$, such that
\begin{equation}\label{subseq}
\text{there are at most $M$ eigenvalues of $\Delta$ in the interval $[n_j^2 - 2K, n_j^2 + 2K]$}.
\end{equation} 
Then for each $n_j$ there is a normalized eigenfunction $u_{k_j}$ such that $\langle u_{k_j}, v_{n_j} \rangle \geq \sqrt{3/4M}$ (choose the normalized  eigenfunction with eigenvalue in the interval $[n_j^2 - 2K, n_j^2 + 2K]$ with the largest component in the direction of  $v_n$; there is at least one eigenfunction with eigenvalue in this range thanks to \eqref{proj}). Then the sequence $(u_{j_k})$ of eigenfunctions has positive mass along bouncing ball trajectories, and in particular is not equidistributed. To see this, given any $\epsilon > 0$, 
let $A$ be a self-adjoint semiclassical pseudodifferential operator, properly supported in the rectangle in both variables, so that $\sigma(A) \leq 1$ and so that $\| (\operatorname{Id} - A) v_n \| \leq \epsilon$ for sufficiently large $n$.
Then, we can compute
\begin{equation*}\begin{gathered}
 \langle A^2 u_{k_j}, u_{k_j} \rangle  = \| A u_k \|^2 \geq \Big|\langle A u_{k_j}, v_{n_j} \rangle\Big|^2 \\
= \Big|\langle  u_{k_j}, A v_{n_j} \rangle\Big|^2 \geq \Big( |\langle  u_{k_j},  v_{n_j} \rangle | - \epsilon \Big)^2
\geq \Big( \sqrt{3/4M} - \epsilon \Big)^2 .
\end{gathered}\end{equation*}
This is bounded away from zero for small $\epsilon$. By choosing  a sequence of operators $A$ such that $\| (\operatorname{Id} - A) v_n \| \to 0$ and such that the support of the symbol of $A$ shrinks to the set of bouncing ball covectors (i.e. multiples of $dy$ supported in the rectangle), we see that the mass of any quantum limit obtained by subsequences of the $u_{k_j}$ must have mass at least $3/4M$ on the bouncing ball trajectories. 

The missing step in this argument, supplied by the present paper (at least for a large measure set in the parameter $t$), is to show that there are indeed sequences $n_j \to \infty$ so that \eqref{subseq} holds.

\begin{remark} Burq and Zworski \cite{BZ} have shown that $o(1)$ quasimodes, unlike $O(1)$ quasimodes, cannot concentrate asymptotically \emph{strictly} inside the rectangle $R = [-\alpha, \alpha] \times [-\beta, \beta]$, in the sense that they cannot concentrate in subrectangles $\omega \times [-\beta, \beta]$ with $\omega$ a strict closed subinterval of $[-\alpha, \alpha]$. 
\end{remark}

\noindent\emph{Acknowledgements.} I wish to thank Maciej Zworski, Steve Zelditch and Alex Barnett  for useful comments on a draft of this manuscript, Harold Donnelly and St\'ephane Nonnenmacher for pointing out several improvements to the first version of this paper, 
and Patrick G\'erard for helpful discussions. 
I especially thank Steve Zelditch for encouraging me to work on this problem several years ago, and for numerous fruitful discussions since that time. 

\section{Hadamard variational formula}\label{Hadamard}
Let $S_t$ denote the stadium billiard with aspect ratio $t$, given explicitly as the union of the rectangle $[-t\pi/2, t\pi/2]_x \times [-\pi/2, \pi/2]_y$ and the circles centred at $(\pm t\pi/2, 0)$ with radius $\pi/2$. 
Let 
 $\Delta_t$ denote the Laplacian on $S_t$ with Dirichlet boundary conditions. Define $E_j(t)$ to be the $j$th eigenvalue (counted with multiplicity) of $\Delta_t$. The key to the proof of Theorem~\ref{ae} for $S_t$ will be a consideration of how $E_j(t)$ varies with $t$. Let $u_j(t)$ denote an eigenfunction of $\Delta_t$ with eigenvalue $E_j(t)$ (chosen orthonormally for each $t$), and let $\psi_j(t)$ denote $E_j^{-1/2}$ times the outward-pointing normal derivative $d_n u_j(t)$ of $u_j(t)$ at the boundary of $S_t$. Let $\rho_t(s)$ denote the function on $\partial S_t$ given by $\rho_t(s) = (\sgn x) \partial_x \cdot n/2$, where $n$ is the outward-pointing unit normal vector at $\partial S_t$. 
 The function $\rho_t$ is the `normal variation' of the boundary $\partial S_t$ with respect to $t$. 
 Notice that $\rho_t \geq 0$ everywhere. 
 
We first observe that the eigenvalue branches $E(t)$ can be chosen holomorphic in $t$. To see this, we fix a  reference domain  $S_1$ and consider
the family of metrics 
$$
g_t = (1 + (t-1) \phi(x))^2 dx^2 + dy^2,
$$
where $\phi(x)$ is nonnegative, positive at $x=0$ and supported
close to $x=0$. If $\int \phi = 1$, then $S_1$ with this metric is isometric to $S_t$ for $1 \leq t \leq 2$. 
Note that $g_t$ is a real analytic family of metrics. Then $\Delta_t$ is (unitarily equivalent to) the Laplacian with respect to the metric $g_t$ on $S_1$.

The analytic family of metrics $g_t$ gives rise to a holomorphic family
of elliptic operators $\tilde L_t$  for $t$ in a complex neighbourhood of $[1,2]$ (with complex coefficients for $t$ non-real), equal to $\Delta_t$ for real $t$. This operator  
acts on $L^2(S_1; dg_t)$ with domain $H^2(S_1) \cap H^1_0(S_1)$,
where $dg_t$ is the measure $(1 + (t-1) \phi(x)) dx dy$. Define the operator $V_t$ by $V_t (f) = (1 + (t-1) \phi(x))^{1/2} f$, which for $t$ real is a unitary operator from $L^2(S_1; dg_t) \to L^2(S_t; dg_1)$. 
Then $\tilde L_t$ is similar to the holomorphic family of operators $L_t = V_t \tilde L_t V_t^{-1}$ acting on $L^2(S_1; dg_1)$ with domain $H^2(S_1) \cap H^1_0(S_1)$, and is unitarily equivalent to $L_t$ for real $t$. 
The family $L_t$ is a holomorphic family of type A in the sense of Kato's book \cite{Kato}. Accordingly, 
the eigenvalues $E(t)$ and eigenprojections can be chosen holomorphic in $t$. Let $u(t)$ be a holomorphic family of eigenfunctions, normalized for real $t$, corresponding to $E(t)$.

\begin{lemma}[Hadamard variational formula] We have
\begin{equation}
\frac{d}{dt} E(t) = -  \int_{\partial S_t} \rho_t(s)(d_n u(t)(s))^2 \, ds.
\label{hv}\end{equation}
\end{lemma}

This is a standard formula (see e.g. \cite{Gara}). It can also be derived from the proof of Proposition~\ref{Edot} using the formula $\dot L_t = [L_t, \partial_x^t \Phi + \Phi \partial_x^t]$ where $\Phi = \int \phi$. 

Now we return to ordering the eigenfunctions $u(t)$ by eigenvalue for each fixed $t$. It follows from holomorphy of the eigenprojections that either $E_j(t) = E_k(t)$ for all $t \in [1,2]$, or $E_j(t) = E_k(t)$ for at most finitely many $t \in [1,2]$. Thus $E_j(t)$ is piecewise smooth and, except for finitely many values of $t$, its derivative satisfies according to \eqref{hv}
\begin{equation}
E_j^{-1} \frac{d}{dt} E_j(t) = -  \int_{\partial S_t} \rho_t(s) \psi_j(s)^2 \, ds.
\label{hvf}\end{equation}

This formula is the basic tool we shall use to prove Theorem~\ref{ae}. 

In Section~\ref{proof}, we will actually prove the following stronger version of Theorem~\ref{ae},  which gives more
information about non-Liouville quantum limits on $S_t$.

\begin{theorem}\label{strong} 
For every $\epsilon > 0$ there exists a subset $B_\epsilon \subset [1,2]$ of measure at least $1-4\epsilon$, and a strictly positive constant $m(\epsilon)$ with the following property. 
For every $t \in B_\epsilon$, there exists a 
quantum limit formed from  Dirichlet eigenfunctions on the stadium $S_t$ that has mass at least $m(\epsilon)$ on the bouncing ball trajectories.  \end{theorem}


\section{The main idea}
Before we give the proof of Theorem~\ref{strong}, we sketch the main idea. For simplicity, in this section we only attempt to argue  that there is at least one $t \in [1,2]$ such that $\Delta_t$ is non-QUE. To do so, let us assume that $\Delta_t$ is QUE for all for $t \in [1,2]$, and seek a contradiction. 

We begin with some heuristics. Let $A(t)$ denote the area of $S_t$. By Weyl's law, we have $E_j(t) \approx c A(t)^{-1} j$. Therefore, since the area of $S_t$ grows linearly with $t$, we have $\dot E_j \approx -\operatorname{const} A(t)^{-1} E_j $, on the average. \emph{The QUE assumption implies that this is true, asymptotically, at the level of each individual eigenvalue.} Indeed, let
\begin{equation}
f_j(t) = \int_{\partial S_t} \rho_t(s) |\psi_j(t; s)|^2 \, ds.
\label{fj-defn}\end{equation}
Then \eqref{hvf} says that $\dot E_j = -E_j f_j$, while the QUE assumption implies that the boundary values $|\psi_j(t)|^2$ tend weakly to $A(t)^{-1}$ on the boundary $\partial S_t$ \cite{GL}, \cite{HZ}, \cite{Burq}. In particular,  this shows that 
\begin{equation}
f_j(t) \to  k A(t)^{-1} > 0, 
\label{kt}\end{equation}
where $k = \int_{\partial S_t} \rho_t(s) \, ds > 0$ is independent of $t$. So, this gives 
\begin{equation}
E_j^{-1} \dot E_j = -k A(t)^{-1} (1 + o(1)), \quad j \to \infty.
\label{reg}\end{equation}
In particular, the magnitude of $E_j(t)^{-1} \dot E_j(t)$ is bounded below for large $j$. This prevents the eigenvalues conspiring to concentrate in intervals $[n^2 - a, n^2 + a]$. Indeed, such concentration, for every $t \in [1,2]$, would require that at least some eigenvalues `loiter' near $E = n^2$ for significant intervals of time $t$, which is ruled out by \eqref{reg}. 
The Heller-O'Connor-Zelditch argument from the Introduction then gives a contradiction to the QUE assumption. 

Rather than employing such a contradiction argument, however, we use a slightly more elaborate direct approach, which yields more information. 

\section{Proof of  Theorem~\ref{strong}}\label{proof}
We begin by dividing  the interval $[1,2]$ into two sets $Z_1 \cup Z_2$, where $Z_1$ is the set of $t$ such that 
\begin{equation}
\liminf_{j \to \infty} f_j(t) = 0,
\label{liminf=0}\end{equation}
where $f_j(t)$ is defined in \eqref{fj-defn}, and $Z_2$ is the complement (i.e. where the lim inf above is positive).

First, consider any $t \in Z_1$. Consider the semiclassical measures $\nu$ on the unit ball bundle of $\partial S_t$ studied in \cite{GL}. 
The relation \eqref{liminf=0} implies that there exists a $\nu$ which vanishes on the curved sides of the stadium. Such a $\nu$ cannot have mass on the boundary of the unit ball bundle,  since the straight part of the boundary is non-strictly gliding \cite{BG}. 
The relation between quantum limits $\mu$ and boundary measures $\nu$ in Theorem 2.3  of \cite {GL} then shows that there exists a  quantum limit $\mu$ supported on (interior) rays that do not meet the curved sides of the stadium.  The only such trajectories are the bouncing ball trajectories. Therefore, every $t \in Z_1$ satisfies the conditions of the theorem.

Next consider $t \in Z_2$. Given $\epsilon > 0$, there is a subset $H_\epsilon$ of $Z_2$, whose measure is at least  $|Z_2| - \epsilon$, such that 
$$
t \in H_\epsilon \implies \liminf_{j \to \infty} f_j(t) \geq c > 0,
$$
where $c$ depends on $\epsilon$. To see this, consider the sets $Z_2^n = \{ t \in Z_2 \mid \liminf f_j(t) \geq 1/n \}$. This is an increasing family of sets whose union is $Z_2$, so by countable additivity of Lebesgue measure, $|Z_2^n| \to |Z_2|$. 
In the same spirit, there is a subset $G_\epsilon$ of $H_\epsilon$, whose measure is at least  $|Z_2| - 2\epsilon$,  where this statement is uniform in $t$; in particular, there exists $N = N(\epsilon)$ such that 
$$
t \in G_\epsilon, \  j \geq N \implies  f_j(t) \geq \frac{c}{2}.
$$

Now we want to consider, for $t \in G = G_\epsilon$, the number of eigenvalues $E_j(t)$ in the interval $[n^2-a, n^2 + a]$. For a fixed $t$, it seems very difficult to improve on the bound $O(n)$ from the remainder estimate in Weyl's law. However, as we see below, one does much better by averaging in $t$. Thus, we shall give a good estimate on 
\begin{equation}
\int_{G_\epsilon} \big( N_t(n^2 + a) - N_t(n^2 - a) \big) \, dt
\label{est2}\end{equation}
for large $n$, where $N_t$ is the eigenvalue counting function for $\Delta_t$.  
This integral can be calculated by considering how much `time' $t$ each eigenvalue $E_j(t)$ spends in the interval $[n^2 - a, n^2 + a]$. 
By Weyl's Law, we have  $\gamma j \leq E_j(t) \leq \Gamma j$ for $t \in [1,2]$,  with $\gamma, \Gamma$ independent of $t$. Therefore, taking $n$ large enough so that $a \leq n^2/2$, we only need consider $j$ such that 
$n^2/2\Gamma \leq j \leq 3n^2/2\gamma$.
Thus, \eqref{est2} is equal to
\begin{equation}
\sum_{j = n^2/2\Gamma}^{3n^2/2\gamma} \Big|\{ t \in G_\epsilon \mid E_j(t) \in [n^2 - a, n^2 + a) \} \Big|. 
\label{j-sum2}\end{equation}

Next, we replace $G = G_\epsilon$ by an open set  containing $G$. 
On $G$ we have $f_j(t) \geq c/2$ for $j \geq N$.  Then the open set 
$$
O_n = \{ t \mid f_j(t) > c/4 \text{ for } N \leq j \leq 3n^2/2\gamma \}
$$ 
contains $G$.  Then for $t \in O_n$, and $n^2 \geq 2\Gamma N$, we have by \eqref{hvf}
\begin{equation}
-\dot E_j (t) \geq c E_j(t)/4, \quad \frac{n^2}{2\Gamma} \leq j \leq \frac{3n^2}{2\gamma}.
\label{4}\end{equation}
Integrating this, we find that for $t_1 < t_2$ in the same component of $O_n$, and $n^2/2\Gamma \leq j \leq 3n^2/2\gamma$,
\begin{equation}
E_j(t_1) - E_j(t_2) \geq \frac{c}{4} E_j(t_2) (t_2 - t_1) \implies 
t_2 - t_1 \leq \frac{4}{c} \frac{E_j(t_1) - E_j(t_2)}{E_j(t_2)}.
\label{diff}\end{equation}
Since $S_t$ is an increasing sequence of domains, the Dirichlet eigenvalues $E_j(t)$ are nonincreasing in $t$. Therefore  \eqref{diff} on each component of $O_n$ implies that  the quantity \eqref{j-sum2}, and hence \eqref{est2},
can be bounded above for $n^2 \geq \max(2a, 2\Gamma N)$ by 
\begin{equation}
\sum_{j = n^2/2\Gamma}^{3n^2/2\gamma} 2a \cdot \frac{4}{c} \cdot \frac1{n^2 - a}  \leq \sum_{j = 1}^{3n^2/2\gamma} 2a \cdot \frac{4}{c} \cdot \frac1{n^2/2}
= \frac{24a}{c\gamma}   .
\label{sum-bound2}\end{equation}

Therefore, on a set $A_n \subset G$ of measure at least $|G| - \epsilon \geq |Z_2| - 3\epsilon$, we can assert that $N_t(n^2 + a) - N_t(n^2 - a)$ is at most $\epsilon^{-1}$ times the right hand side of \eqref{sum-bound2}. That is, for  sufficiently large $n$, there is a set $A_n$ of measure at least $|Z_2| - 3\epsilon$ on which 
$$
N_t(n^2 + a) - N_t(n^2 - a) \leq \frac{24a}{c \gamma \epsilon}  ,
$$
which is a bound manifestly independent of $n$. 

To finish the proof we show that there is a set of measure at least $|Z_2| - 4\epsilon$ that is contained in $A_n$ for infinitely many $n$. 
That is, defining
\begin{equation}\label{bk}
B_k = \{ t \in Z_2 \mid t \in A_n \text{ for at least $k$ distinct values of $n$} \},
\end{equation}
we show that $|\cap_k B_k| \geq |Z_2| - 4\epsilon$. To show this consider the sets 
\begin{equation*}
D_k = \{ t \in Z_2 \mid t \in A_n \text{ for at least $k$ distinct values of $n$ in the range $k \leq n < 5k$} \}. 
\end{equation*}
Since $D_k \subset B_k$ and $B_k$ is a decreasing family of sets, it suffices to show that $|D_k| \geq |Z_2| - 4\epsilon$ for every $k$.  
To see this, on one hand, we have 
$$
\sum_{n=k}^{5k-1} |A_n| \geq 4k(|Z_2| - 3\epsilon).  
$$
On the other hand, by the definition of $D_k$,
$$
\sum_{n=k}^{5k-1} |A_n| \leq 4k |D_k| + k (|Z_2| - |D_k|),
$$
and putting these together we obtain
$$
|D_k| \geq |Z_2| - 4\epsilon,
$$
as required. 

We have now shown that for a subset of $Z_2$ of measure at least $|Z_2| - 4\epsilon$, there is a sequence of integers $n_j$ (depending on $t$) for which \eqref{subseq} holds, and therefore the mass statement in Theorem~\ref{strong} holds for all such $t$ using the argument from the Introduction. Thus the conclusion of Theorem~\ref{strong} holds for all $t \in Z_1$ and all $t \in Z_2$ except on a set of measure at most $4\epsilon$.  This completes the proof.

\section*{Appendix. By Andrew Hassell and Luc Hillairet}
In this appendix, we show how the proof above for the stadium domain can be adapted to partially rectangular domains $X_t$, thereby proving Theorem~\ref{ae} in full generality. Again we prove a stronger version
which gives more
information about non-Liouville quantum limits on $X_t$. To state this result, we denote by $BB$ the union of the bouncing-ball trajectories in $S^*X_t$, by $TT$ the union of billiard trajectories that do not enter the rectangle (`trapped trajectories'), and by $ET$ the excluded trajectories in \cite{ZZ}. The set $ET$, only relevant when $X_t$ has boundary, consists of the billiard trajectories that either (i) hit a non-smooth point of the boundary at some time, (ii) reflect from the boundary infinitely often in finite time, or (iii) touch $\partial X_t$ tangentially at some time\footnote{Here we do not exclude the trajectories that do not meet the boundary forwards or backwards in time as is done in \cite{ZZ}.}. 
All these sets have measure zero; that $TT$ has measure zero follows from ergodicity, while that $ET$ has measure zero is shown in \cite{ZZ}. 

\begin{theorem}\label{strong-prd} Let $X_t$ be a partially rectangular domain, and let $\Delta_t$ be the Laplacian on $X_t$. For every $\epsilon > 0$ there exists a subset $B_\epsilon \subset [1,2]$ of measure at least $1-4\epsilon$, and a strictly positive constant $m(\epsilon)$ with the following property. 
For every $t \in B_\epsilon$, there exists a quantum limit of $\Delta_t$ that either has mass at least $m(\epsilon)$ on $BB$, or else 
 concentrates entirely on $BB \cup TT \cup ET$.  \end{theorem}

\begin{remark} Since $BB \cup TT \cup ET$ has measure zero, this implies that $\Delta_t$ is non-QUE.
\end{remark}

The main task is to replace the boundary formula \eqref{hvf} for the variation of eigenvalues with an interior formula. 
Let $X$ be a partially rectangular domain with rectangular part $[-\pi/2, \pi/2] \times [-\pi/2, \pi/2]$, and let $X_t$ be the domain with the rectangle replaced by $[-t\pi/2, t\pi/2] \times [-\pi/2, \pi/2]$, where $t \in [1,2]$. 
Let $\Delta_t$ denote the Laplacian on $X_t$ with either the Dirichlet boundary condition or the Robin boundary condition  $d_n u = bu$, $b \in \RR$ constant (which of course includes the Neumman condition as the special case $b=0$).

We now compute the variation of the eigenvalues of $\Delta_t$
with respect to $t$. To state this result, we introduce some notation.
Let $g_t$ and $L_t$ be as in Section~\ref{Hadamard}, and let $I_t$ denote the isometry from $(X_1, g_t)$ to $X_t$. 
Let $M_t$ denote the multiplication operator $(1 + (t-1) \phi(x))^{-1/2}$,
and let $\partial_x^t$ denote $M_t \partial_x M_t$. Then 
$L_t$ is given by $-(\partial_x^t)^2 - \partial_y^2 $ on its domain. Let $\phi_t$ denote the function $\phi M_t^2 \circ I_t^{-1}$ on $X_t$.

\begin{proposition}\label{Edot} Let $u(t)$ be a real eigenfunction of $\Delta_t$, $L^2$-normalized on $X_t$, with eigenvalue $E(t)$, depending smoothly on $t$.  Let $Q$ be the operator
\begin{equation}
Q = - 4 \partial_x \phi_t \partial_x + [\partial_x, [\partial_x, \phi_t]]
\label{Q-defn}\end{equation}
acting on functions on $X_t$. Then 
\begin{equation}
\dot E(t) =  - \frac1{2}\langle Q u(t),  u(t) \rangle 
\label{dotE}\end{equation}
and there exists $C$, depending only on the function $\phi$, such that
\begin{equation}
\dot E(t) \leq C, \quad t \in [1, 2].
\label{dotEleqC}\end{equation}
\end{proposition}

\begin{proof} 
Let $v(t)$ be the eigenfunction of $L_t$ corresponding to $u(t)$. Then 
$$
E(t) = \langle L_t v(t), v(t) \rangle.
$$
Since $\dot v(t)$ is orthogonal to $v(t)$, while $L_t v(t)$ is a multiple of $v(t)$, we have
\begin{equation}
\dot E(t) = \langle \dot L_t v(t), v(t) \rangle.
\label{dotE1}\end{equation}
Using the expression for $L_t$ above, we have
$$
\dot L_t = \partial_t \big( \partial_x^t)^2,
$$
and since
$$
\partial_t \partial_x^t = -\frac1{2} \Big( \phi M_t^2 \partial_x^t +   \partial_x^t M_t^2 \phi \Big),
$$
we obtain
$$
\dot L_t = -\frac1{2} \Big(\phi M_t^2 (\partial_x^t)^2 +      2 \partial_x^t \phi M_t^2 \partial_x^t +  (\partial_x^t)^2 \phi M_t^2 \Big).
$$
Substituting this into \eqref{dotE1} gives an expression for $\dot E(t)$ in terms of $v(t)$. Writing this in terms of $u(t)$ on $X_t$ gives the equivalent expression
\begin{equation}
\dot E(t) = - \frac1{2}  \Big\langle \Big(\phi_t \partial_x^2 +      2 \partial_x \phi_t  \partial_x +  \partial_x^2 \phi_t \Big) u(t), u(t) \Big\rangle
\label{EEE}\end{equation}
which can be rearranged to \eqref{dotE}. To prove \eqref{dotEleqC}, we observe that $- 4 \partial_x \phi_t \partial_x$ is a positive operator, while $[\partial_x, [\partial_x, \phi_t]]$ is a multiplication operator by a smooth function of $x$ and $t$, and hence bounded as an operator on  $L^2$ by a constant independent of $t$ for $t \in [1,2]$. 
\end{proof}

Now we indicate how the proof in Section~\ref{proof} may be modified to prove Theorem~\ref{strong-prd}. We redefine $f_j(t)$ by
\begin{equation}
f_j(t) = E_j^{-1} \langle Qu_j, u_j \rangle
\label{fj-defn2}\end{equation}
so that $\dot E_j(t) = - E_j(t) f_j(t)$ as above, and partition the $t$-interval $[1,2]$ into $Z_1 \cup Z_2$ as before. Consider any $t \in Z_1$. Then there is an increasing sequence $j_k$ of integers such that $f_{j_k}(t) \to 0$. In this case, we can construct an operator properly supported in the interior of $X_t$ for which \eqref{qe} fails to hold for $j = j_k \to \infty$. 
Choose a function $\zeta(y)$ taking values between 0 and 1 which is equal to $1$ near $y=0$ and vanishes near $|y| = \pi/2$. Then in view of \eqref{fj-defn2} and \eqref{Q-defn},
\begin{equation}\begin{gathered}
\| \zeta(y) (\phi^t)^{1/2} (h\partial_x) u_{j_k} \|_2^2 \leq 
\| (\phi^t)^{1/2} (h\partial_x) u_{j_k} \|_2^2 \\
= -\langle  h^2 Q u_{j_k},  u_{j_k} \rangle + O(h)
\to 0, \quad h = h_{j_k} = E_{j_k}^{-1/2}. 
\end{gathered}\end{equation}
Therefore, defining $A_h = \zeta(y)^2 (h \partial_x) \phi^t(x) (h \partial _x)$, we have
$$
\lim_{k \to \infty} \langle A_{h} u_{j_k}, u_{j_k} \rangle = \| \zeta(y) (\phi^t)^{1/2} (h\partial_x) u_{j_k} \|_2^2 \to 0 \neq \frac1{S^*X_t} \int_{S^*X_t} \sigma(A_h), \quad h = h_{j_k} ,
$$
since $\sigma(A_h) \geq 0$, and is $> 0$ on a set of positive measure. 
Thus $X_t$ is not QUE. 
Moreover, using the pseudodifferential calculus,  this  implies that $\langle B_{h_{j_k}} u_{j_k}, u_{j_k} \rangle \to 0$ for any $B_h$ microsupported where $\sigma(A_h) > 0$. Then, parametrix constructions for the wave operator microlocally near rays that reflect nontangentially at the boundary (see for example \cite{Chaz}) show that $\langle B'_{h_{j_k}} u_{j_k}, u_{j_k} \rangle \to 0$ for any $B'_h$ with symbol supported close to any $q'  \in S^* X_t^\circ$ enjoying the property that it is obtained from a point $q$ such that $\sigma(A_h) (q) > 0$ by following the billiard flow through a finite number of nontangential reflections at smooth points of $\partial X_t$. 
This property is true for any  $q' \notin BB \cup TT \cup ET$ (for a suitable choice of $\zeta$ depending on $q'$), since the symbol of $A_h$ is positive on all unit covectors lying over $\supp \phi^t \times \supp \zeta$ which are not vertical. 
It follows that the sequence $u_{j_k}$ concentrates away from all such points, which is to say that it concentrates at $BB \cup TT \cup ET$.

The argument for $t \in Z_2$ goes just as before, except that instead of the nonincreasing condition $\dot E_j(t) \leq 0$, we only have the weaker condition $\dot E_j(t) \leq C$ thanks to \eqref{dotEleqC}. We modify the argument below equation \eqref{diff} as follows: 
Define $ E^*_j(t) = E_j(t) - Ct$; then $\dot E^*_j(t) \leq 0$ and \eqref{4} is valid for $E_j^*(t)$, hence we obtain using the method of Section~\ref{proof}
\begin{equation}
\sum_{j = n^2/2\Gamma}^{3n^2/2\gamma} \Big|\{ t \in G_\epsilon \mid E^*_j(t) \in [n^2 - a^*, n^2 + a^*) \} \Big| \leq  \frac{24a^*}{c\gamma}  .
\label{j-sum2*}\end{equation}
Now we use the observation
$$E_j(t) \in [n^2 - a, n^2 + a) \implies E_j^*(t) \in [n^2 - a - 2C, n^2 + a + 2C)$$
to deduce the estimate \eqref{j-sum2*} for $E_j(t)$ with $a^*$ replaced by $a$ on the left hand side and by $a + 2C$ on the right hand side. 
The rest of the argument from Section~\ref{proof} can now be followed to its conclusion.

\end{document}